\documentclass[12pt]{article}

\usepackage{graphicx}
\usepackage{enumerate}
\usepackage{natbib}

\usepackage{graphicx,bm,url,microtype}
\usepackage{paralist}
\usepackage{authblk}
\usepackage{amsfonts,amssymb,amsthm,amsmath, amsthm}

\usepackage[a4paper,text={16.5cm,25.2cm},centering]{geometry}
\usepackage[compact,small]{titlesec}

\newtheorem{teo}{{\sc\bf Theorem}}
\newtheorem{lem}{{\sc\bf Lemma}}
\newtheorem{rem}{{\sc\bf Remark}}

\newtheorem{defn}{{\sc\bf Definition}}
\newtheorem{cor}{{\sc\bf Corollary}}

\usepackage{amsmath,amsthm,mathrsfs,amsfonts}
\usepackage{ulem} 	
\usepackage{url}
\usepackage{amsfonts,todonotes}
\usepackage{amssymb}
\usepackage{bbm}
\usepackage{graphicx}
\usepackage{subfig}
\usepackage{float}

\newcommand{\eps}{\varepsilon}

\usepackage{multirow}

\usepackage{parskip}
\usepackage{booktabs}
\usepackage[shortlabels]{enumitem}
\usepackage{unicode}

\begin{document}

	\begin{center}
		\Large \bf Universally consistent estimation of the reach
	\end{center}
	\begin{center}
		Alejandro Cholaquidis $^1$, \hspace{.2cm}
		Ricardo Fraiman $^1$  \hspace{.2cm}  and \hspace{.2cm}  Leonardo Moreno$^2$ \\
		$^1$ Centro de Matemáticas, Facultad de Ciencias, Universidad de la República, Uruguay. \\
		$^2$ Instituto de Estadística, Departamento de Métodos Cuantitativos, FCEA, Universidad de la República, Uruguay.
	\end{center}

\begin{abstract}
	The reach of a set $M \subset \mathbb R^d$, also known as condition number when $M$ is a manifold, was introduced by Federer in 1959. The reach is a central concept in geometric measure theory, set estimation, manifold learning,  among others areas. 
	We introduce a universally consistent estimate of the reach, just assuming that the reach is positive. Under an additional assumption we provide rates of convergence. We also show that it is not possible to determine, based on a finite sample, if the reach of the support of a density is zero or not. We provide a small simulation study and a bias correction method for the case when $M$ is a manifold.
\end{abstract}

 \section{Introduction}

 The reach of a set  $M\subset \mathbb{R}^d$, denoted by $\textrm{reach}(M)$, is a key concept in geometric measure theory; see \cite{fed:69,rataj17,rataj01}, and the references therein. It is also of importance   in set estimation. We first focus on its relevance as a geometric concept, and then on the importance of shape constraints in set estimation, and in particular the positive reach condition.

 Is defined as the largest distance from which any point outside $M$ has a unique nearest point in $M$; see Figure \ref{fi1} (see also Definition \ref{reach1}). It can be proved that it is infinity for convex set. 
 Positive reach impose some regularity conditions on the set. For instance, its boundary has Lebesgue measure zero,  as it follows from the fact that the class of sets with positive reach is a subclass of the cone-convex sets (see \cite{cue12} and Proposition 2 in \cite{chola:14}). Moreover, the  volume (i.e., its $d$-dimensional Lebesgue measure) of the set of points at distance $t\leq \textrm{reach}(M)$ is a polynomial of degree $d$ on $t$ for all $t\in [0,\textrm{reach}(M)]$  (see \cite{federer:59}). If $\textrm{reach}(M)>0$, then it is also possible to define its second fundamental form; see \cite{fu}. The positivity of the reach of the boundary, $\partial M$, of $M$ allows its Minkowsky content to be defined (see \cite{amb08}), which is a notion of surface area suitable for estimation purposes (see \cite{cue13}). It have been proved (see for instance Proposition 6.1 in \cite{sma08}) that the reach is an upper bound for the inverse of the curvature of an arc-length parametrized geodesic.  In addition, as mentioned in \cite{aa21} ``it prevents quasi self-intersection at scales smaller than the reach'' (see Theorem 3.4 in \cite{aa19})

 \begin{figure}[ht]
 	\begin{center}
 		\includegraphics[height=3cm,width=4.8cm]{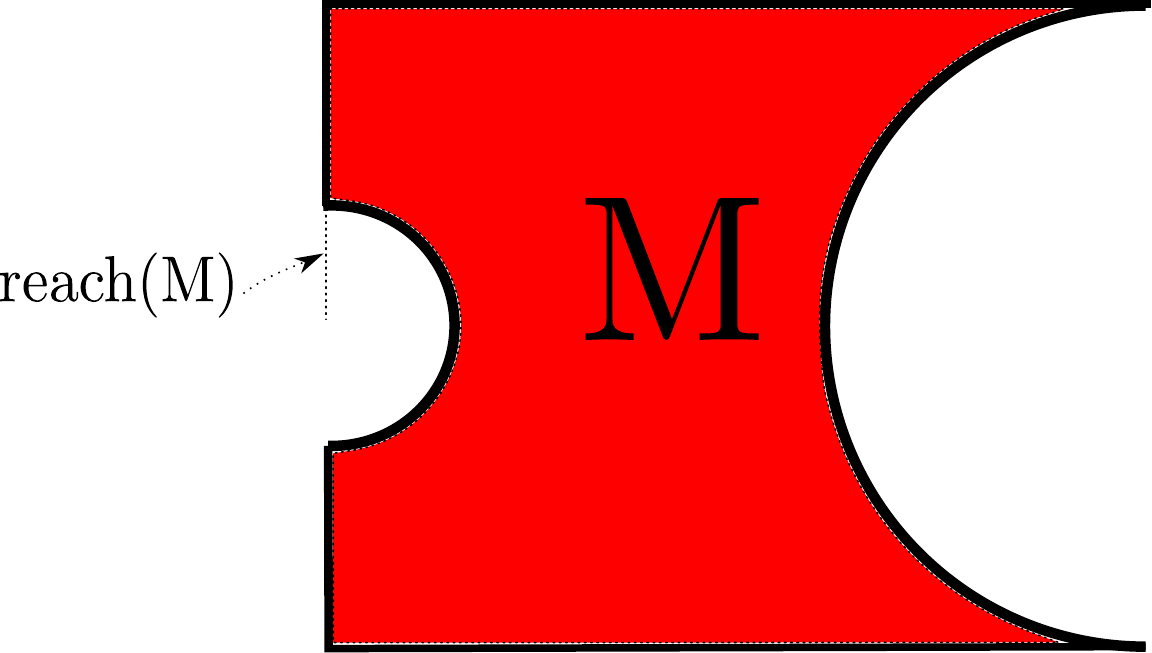}
 	\end{center}
 	\caption{The reach of a set $M$, $\textrm{reach}(M)$ is the largest distance form which there exists a unique nearest point on $M$.}
 	\label{fi1}
 \end{figure}

 When the set is a smooth manifold, the reach parameter is also known as ``condition number'' (see for instance \cite{sma08}).  It can also be proved that  $\mathcal{C}^2$ compact manifolds  with empty or $\mathcal{C}^2$ boundary, have positive reach; see \cite{tha08}. 
 
 The reach has gained importance in the last two decades in the areas of statistics known as set estimation, manifold learning and persistent homology (see for instance \cite{arias2020,chola16},  and \cite{ho19}, respectively). Given an unknown set $M \subset \mathbb R^d$ (not necessarily convex), set estimation deals with the problem of the estimation of $M$ from a random sample $\{X_1, \ldots, X_n\} \subset M$, as well as several geometrically important functional associated with $M$. For instance, its Lebesgue measure, the Minkowsky content of its boundary, among others.
 
 The starting and simplest problem is when $M$ is the support of a distribution.  This problem has been addressed by several authors where the ground-breaking Devroye-Wise estimator is a milestone (see  \cite{dw:80}), due to its universal consistency and simplicity. To obtain rates of convergence for this estimator, some geometric restrictions on the set are required (see \cite{crc:04}). The best attainable rate for this estimator is of the order $\mathcal{O}(\log(n)/n)^{1/d}$.  
 
 To improve the rates of convergence, it is necessary to impose stronger shape restrictions than the one required in \cite{crc:04} on the set. A well-known shape restriction is $r$-convexity. A set is said to be $r$-convex if it equals the complement of open balls of radius $r$ not meeting the set. $r$-convexity is one of the most studied shape restriction in set estimation (see for instance, \cite{pateiro:09,cas07,cas16}, and the references therein).  
 
 Positive reach is a stronger condition than $r$-convexity (see \cite{cue12}). Therefore, when a ball of positive radius rolls inside the set (i.e there exists $\lambda>0$ such that for all $x\in \partial M$,  there exists $y\in M$ such that $x\in \partial \mathcal{B}(y,\lambda)\subset M$), and the set has reach $r$, the estimator proposed in \cite{cas21} estimates the reach. Also, if the set and its complement are $r$-convex, then the set has positive reach  (see Lemma A.0.7 in \cite{beatesis} together with \cite{wal99}).  
 
 A second stage in set estimation is to estimate the level sets of the distribution (see \cite{cad06,cgmrc06}), as well as the boundary of the set (see \cite{crc:04}). In this setting, \cite{cf97} proposes to estimate the support of the distribution by means of a kernel-based density estimator.  
 
 Next, and getting closer to our problem, the interest is on some functional of the set, such as the $d$-dimensional Lebesgue measure of the set, as well as the measure of its boundary; see for instance \cite{cfrc07}.
 
 Recently, the study of statistical methods for manifold valued data (known as manifold learning) has gained attention, due to its application in dimension reduction  (among others). The aim is to recover a lower dimensional structure from the data, see for instance \cite{aa18,fef16,gen12a,gen12b,sma08}and the references therein, or a functional of it, see for instance \cite{aa20,aa17}. Several classical problems have been tackled in this setting, such as density estimation. Here again, the reach  plays a key role as a shape restriction. In \cite{aa19}, an estimator of the reach is proposed for manifold valued data with ``the key assumption that both the point cloud and the tangent spaces were jointly observed''.  For this ``oracle framework [...] it is showed to achieve uniform expected loss bounds over a $\mathcal{C}^3$-like model" and ``upper and lower bounds on the minimax rate for estimating the reach are obtained''.  Also for manifold valued data a different estimator of the reach was introduced recently in \cite{bere22}, where the manifold is assumed to be at least of class $\mathcal{C}^3$, and it must be previously estimated with the manifold estimator proposed in \cite{aa19_2}. The rates of convergence obtained (in probability and in $L^1$) are  of order $(\log(n)/n)^{k/(2d)}$  if the manifold $M\subset \mathbb{R}^d$ is of class $C^k$, for $k\geq 4$, and $(\log(n)/n)^{1/d}$, for $k=3$. The two aforementioned estimators requires the manifold to have no boundary. This assumption is not required in our proposal, which is, up to our knowledge, the only consistent estimator proposed in this setup.   In  \cite{aa22} an estimator of the reach is also proposed. They obtain better convergence rates,   but with stronger hypotheses and a non-computable estimator.

 In what follows, we will study the problem of estimating the reach of a set, looking for universally consistent  estimators, that is, assuming only that the set has positive reach. The rate of convergence obtained depends on the Hausdorff distance between the sample and the set. No assumptions are made on the distribution of the sample.

 Our estimation procedure is based on an equivalent definition of reach given in  \cite{boi}, which provides a new nice geometrical interpretation of the reach. In Theorem \ref{consist}, we prove the universal consistency of the estimator.  With an additional assumption in Corollary \ref{cor}, we derive a convergence rate for the proposed reach estimator. In Section \ref{nonest}, we prove that it is not possible to determine based on a finite sample if the reach is zero or not.  In
 Section \ref{simus} we report the results of a small simulation study, and in 
 Subsection \ref{sesgo}, we introduce a bias correction method for the case where $M$ is a manifold.

 \section{Main definitions and geometric results}
 We start by fixing some notation to be used throughout the manuscript. Given a set $M\subset\mathbb{R}^d$, we denote by $\partial M$ and $\textnormal{int}(M)$, the boundary and  interior of $M$, respectively.  In what follows, we assume that $M$ is compact.
 
 We denote by $\left\|\cdot\right\|$ the Euclidean norm in $\mathbb{R}^d$.
 A closed ball of radius $\eps>0$ centred at $x$ is denoted by $\mathcal{B}(x,\eps)$, and an open ball is denoted by $\mathring{\mathcal{B}}(x,\eps)$.  
 Given $\eps>0$ and a set $A\subset \mathbb{R}^d$, $B(A,\eps)$ denotes the parallel set $B(A,\eps)=\{x\in \mathbb{R}^d\colon\ d(x,A)\leq \eps\}$, where $d(x,A)=\inf\{\|x-a\|\colon\ a\in A\}$. The $d$-dimensional Lebesgue measure on $\mathbb{R}^d$ of a set $M$ is denoted by $|M|_d$.
 Given two non-empty compact sets $A,C\subset \mathbb{R}^d$, the Hausdorff distance between $A$ and $C$ is defined as 
 $$d_H(A,C)=\max\big\{\max_{a\in A}d(a,C), \ \max_{c\in C}d(c,A)\big\}.$$ Given a continuous  curve $\gamma:[0,T]\to M$, we define its length as
 $$l(\gamma)=\sup_P \sum_i ||\gamma(t_{i+1})-\gamma(t_i)||,$$
 where the supremum is over all finite partitions of $[0,T]$. Given  $x,y\in M$ we define the geodesic distance between them as $d_M(x,y)=\inf_\gamma l(\gamma)$,
 where the infimum is over all continuous curves joining $x$ and $y$. In what follows, we assume that $M$ is geodesically convex; that is, for any two points $x,y\in M$ there exists a geodesic connecting them, with length $d_M(x,y)$ (see \cite{bern}).

 
 Following the notation in \cite{federer:59}, let $\text{Unp}(M)$ be the set of points $x\in \mathbb{R}^d$ with a unique closest point on $M$.

 \begin{defn}\label{reach1} For $x\in M$, let {\emph{reach}}$(M,x)=\sup\{r>0:\mathring{\mathcal{B}}(x,r)\subset\emph{Unp}(M)\big\}$. The reach of $M$ is defined by $\emph{reach}(M)=\inf\big\{\emph{reach}(M,x):x\in M\big\},$ and $M$ is said to be of positive reach if $\emph{reach}(M)>0$.
 \end{defn}

 Theorem 1 in \cite{boi}, states that  when $M$ is closed $\textrm{reach}(M)$ equals
 \begin{equation}\label{reach2}
 	\sup \Big\{  r>0, \ \forall a, b \in M, \Vert a-b\Vert < 2r \Rightarrow  d_M(a,b)\leq 2r \text{arcsin}\left(\frac{\Vert a-b\Vert}{2r}\right)\Big\}.
 \end{equation}

 Another important geometric restriction, which is required to get the convergence rate of the estimator in the iid case, is the standardness, see \cite{crc:04}.

 \begin{defn}\label{def-stand}  A set $M\subset \mathbb{R}^d$ is said to be standard with respect to a
 	Borel measure $\nu$ in a point $x$ 
 	if there exists $\lambda>0$ and $\eta>0$ such that
 	\begin{equation} \label{estandar}
 		\nu(\mathcal{B}(x,\eps)\cap M)\geq \eta |\mathcal{B}(x,\eps)|_d,\quad 0<\eps\leq \lambda.
 	\end{equation}
 	A set $M\subset \mathbb{R}^d$ is said to be standard if \eqref{estandar} hold for all $x\in M$.
 \end{defn}

 Let $\mathcal{X}_n=\{X_1,\dots,X_{n}\}\subset M$ be a finite set, and $G_n=(\mathcal{X}_n,E_n)$ be a graph with vertex in $\mathcal{X}_n$ and edges $E_n$.
 
 Define as in \cite{bern}
 \begin{align*}
 	d_{G_n}(x,y)=&\min_P \sum_{i=1}^{p-1} ||X_{j_{i+1}}-X_{j_i}||,
 \end{align*}	
 where $P=(X_{j_{1}},\dots,X_{j_{p}})\subset\mathcal{X}_n$ varies over all paths along the edges of $G_n$ connecting $x=X_{j_1}$ to $y=X_{j_p}$.
 
 \begin{defn}
 	Let $\{\eps_n\}_n$ a sequence of strictly positive real numbers such that $\eps_n\to 0$, let us define the graph $G_n=(\mathcal{X}_n,E_n)$ such that $(X_i,X_j)\in E_n$ if and only if $||X_{i}-X_j||\leq \eps_n$. 
 \end{defn}

 The following theorem, whose proof is given in the Appendix, states that $d_{G_n}$ approximates $d_M$.  The proof follows closely the ideas used in \cite{bern}. It is included because we state our results in terms of the reach of the set, and not the curvature, as it is done in \cite{bern}, and the points were there are changes are clarified.

 \begin{teo} \label{auxth} With the notation introduced before, let $M$ be a compact, geodesically convex set, such that its reach, $r_0$, is strictly positive. Let $\{\eps_n\}_n$ be a sequence of strictly positive real numbers such that  $\tau_n=d_H(\mathcal{X}_n,M)/\eps_n\to 0$,  where $\eps_n\to 0$,  and $n$ is large enough to guarantee $\eps_n<2r_0$. Then, for all $x,y\in M$,
 	\begin{equation}\label{auxeq1} 
 		\Big(1- \frac{1}{24}\Big(\frac{\pi \eps_n}{2r_0}\Big)^2\Big) d_M(x,y)  \leq d_{G_n}(x,y)\leq (1+4\tau_n)d_M(x,y),
 	\end{equation}	
 	where  $d_{G_n}(x,y)$ is the distance in the graph build previously, but including $x$ and $y$ as vertices.
 \end{teo}

 \section{Reach estimation}
 \label{nuestroe}
 Given a set $\mathcal{X}_n=\{X_1, \ldots, X_n\} \subset M$,  and a sequence of positive real numbers $\{\eps_n\}_n$, the plug-in estimator of $\textrm{reach}(M)$ based on the graph $G_n$ defined in the previous section, is
 
 \begin{multline*}
 	\hat{r}_n=\sup \Big\{  r>0, \ \forall X_i \neq X_j \in \mathcal{X}_n,  \Vert X_i-X_j\Vert < 2r \Rightarrow \\  d_{G_n}(X_i,X_j)\leq 2r(1+\eps_n^2)\text{arcsin} \Big(\frac{\Vert X_i-X_j\Vert  }{2r }\Big)\Big\}.
 \end{multline*}

 In Figure \ref{fir}, it is shown as a green dotted line $d_{G_n}(a,b)$, the distance between two points $a,b\in M$ (with $\Vert a-b \Vert < 2r $) in the graph $G_n$ built as before for a given $\eps_n$. As a black solid line, it is shown an arc in the circle of radius $r$ joining $a$ and $b$ ($a,b \in \mathcal{X}_n$), whose length is $2r \text{arcsin}(\Vert a-b\Vert/2r)\geq d_{G_n}(a,b)$.

 \begin{figure}[ht]
 	\begin{center}
 		\includegraphics[scale=2]{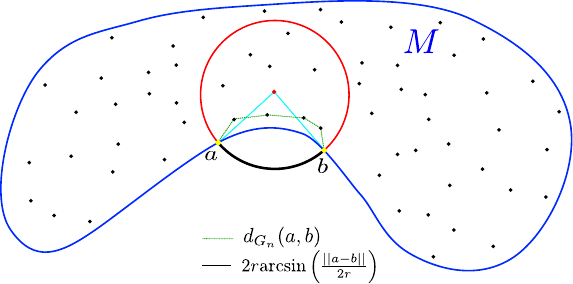}
 	\end{center}
 	\caption{As a green dotted line, it is shown $d_{G_n}(a,b)$. As a black solid line, it is shown as the arc of length  $2r \text{arcsin}(\Vert a-b \vert /2r)$ joining two points $a,b \in \mathcal{X}_n$. }
 	\label{fir}
 \end{figure}
 
 The following theorem states that the estimator $\hat{r}_n$ is, for $n$ large enough, bounded from below by the reach $r_0$ of $M$, and bounded from above by $r_0/(1-\eps_n)$. 
 
 \begin{teo} \label{consist} Under the hypotheses of Theorem \ref{auxth}, for all $n$ large enough,  we have
 	\begin{equation}\label{eqth}
 		r_0 \leq \hat{r}_n \leq \frac{r_0}{1-\eps_n}, 
 	\end{equation}
 	where $\eps_n$ is a sequence of strictly positive real numbers such that $\eps_n\to 0$ and $\delta_n/\eps_n^3\to 0$, being $\delta_n=d_H(\mathcal{X}_n, M)$.
 \end{teo}

 If we assume that $\textrm{reach}(M)>0$, and $M$ is standard (see Definition \ref{def-stand}), we get the following corollary, regarding the convergence rate of $\hat r_n$, which is a consequence of Theorem 3 in \cite{crc:04} and Theorem \ref{consist}.

 \begin{cor} \label{cor} Let $X_1,X_2,\dots$ be a sequence of iid observation drawn from a distribution $P_X$ on $\mathbb{R}^d$. Assume that the support $M$ of $P_X$ is compact, standard with respect to $P_X$, and has reach $r_0>0$. Then,  with probability one, for $n$ large enough,
 	\begin{equation}\label{rate1}
 		r_0\leq \hat{r}_n\leq r_0+c\Big(\frac{\log(n)}{n}\Big)^{\frac{1}{3d}}\beta_n,
 	\end{equation}
 	for $\eps_n=c(\log(n)/n)^{\frac{1}{3d}}\beta_n$, being  $c$ any constant larger than $(2/(\eta \omega_d))^{1/d}$, where $\omega_d=|\mathcal{B}(0,1)|_d$, and $\eta$ is the standardness constant, and $\beta_n\to +\infty$ any sequence.
 \end{cor}
 
 \begin{rem}
 	The standarness hypothesis (defined in \eqref{estandar}) in Corollary \ref{cor} is not fulfilled when $M$ is a $d'$-dimensional Riemannian manifold, $d'<d$, and $P_X$ is a distribution supported on $M$. In that case, instead of assumption \eqref{estandar}, the required condition for $P_X$ and $M$ is the following: there exists $\lambda>0$ and $\eta>0$ such that for all $x \in M$ $P_X(\mathcal{B}(x,\eps))\geq \eta|\mathcal{B}(x,\eps)|_{d'}$ for all $0<\eps<\lambda$,  $\mathcal{B}(x,\eps)$ being  the ball w.r.t. the Riemannian metric, and $|\cdot|_{d'}$ being the $d'$-dimensional Lebesgue measure. With the same ideas used to prove Theorem 3 in \cite{crc:04}, it can be proved that $d_H(\aleph_n,M)=\mathcal{O}((\log(n)/n)^{1/d'})$, and then the rate \eqref{rate1} is improved up to 
 	$$r_0\leq \hat{r}_n\leq r_0+c(\log(n)/n)^{1/(3d')}\beta_n,$$ 
 	for $\eps_n=c(\log(n)/n)^{\frac{1}{3d'}}\beta_n$, being $c$ any constant larger than $(2/(\eta \omega_{d'}))^{1/d'}$, where $\omega_{d'}$ is the $d'$-dimensional Lebesgue measure of a ball of radius one in $\mathbb{R}^{d'}$, $\eta$ is the standardness constant, and $\beta_n\to +\infty$ is any sequence of positive real numbers. 
 	
 \end{rem} 
 \color{black}


 \section{Non-estimability}\label{nonest}
 
 In this section we will prove that it is not possible to determine, based on  a finite sample, if the reach of the  support, $sup(f)$, of a density $f$ is zero or not. 
 This is equivalent to showing that the functional
 $$\alpha(f)=\begin{cases}
 	1 & \text{ if }\textrm{reach}(sup(f))=0\\
 	0 & \text{otherwise},
 \end{cases}
 $$
 cannot be consistently estimated from a sequence of iid observations. This is the case of several different problems. The most simple and well-known is that one cannot determine if the mean of a distribution is finite or not based on a finite sample. However, as in our case, once assuming that the mean is finite, it can be consistently estimated. A fundamental contribution on this general problem is given in \cite{lecam}, where they provide necessary and sufficient conditions for the existence of consistent estimators.

 To prove that $\alpha(f)$ cannot be consistently estimated, we will make use Lemma 1.1 in \cite{fraimeloche}. 
 \begin{lem}[\textbf{Lemma 1.1 in \cite{fraimeloche}}] \label{fm}
 	Let $\mathcal T$ a subset of densities equipped with some norm $\Vert . \Vert$ that makes $(\mathcal{T}, \Vert . \Vert)$ a complete metric space. Assume that $\Vert f \Vert_1 \leq c \Vert f \Vert$ for some constant $c$ and all $f \in \mathcal T$. Let $\Phi: \mathcal T \to [-K,K]$ be any bounded characteristic of the densities in $\mathcal T$. If $\Phi$ is consistently estimable on $\mathcal{T}$, then there exists a dense subset of points in $\mathcal T$ at which $\Phi$ is continuous with respect to the topology induced by the $\Vert . \Vert$ norm. Therefore, if $\Phi$ is discontinuous at every point in $\mathcal T$, it is not consistently estimable in $\mathcal T$.
 \end{lem}
 In our case, $\mathcal{T}$ is the set of densities (w.r.t the Lebesgue measure) endowed with the $L^1$ norm, $\|\cdot\|_1$, which is, by Scheffe's lemma, a complete space.
 
 \begin{teo}
 	Let $f$ be a density such that $\emph{reach}(supp(f))>0$, then  for all $\delta>0$ there exists $f_\delta$ a density such that $||f-f_\delta||_1<\delta$ and $\emph{reach}(sup(f_\delta))=0$. Moreover,  if $\emph{reach}(sup(f))=0$, then for all $\eps>0$ there exists $f_\eps$ such that $\|f_\eps-f||_{1}<\eps$ and $\emph{reach}(sup(f_\eps))>0$. Thus, the functional $\alpha$ is discontinuous at any density in $\mathcal{T}$, and therefore from    Lemma \ref{fm}, it is not estimable.
 \end{teo}
 \begin{proof}  If $\textrm{reach}(sup(f))>0$ then from Propositions 1 and 2 in \cite{cue12} $sup(f)$ fulfills the exterior rolling ball condition (see section 2 in \cite{cue12}). It is easy to see that this imply that $sup(f)$ is $\rho,h$-cone-convex for some $\rho\in (0,\pi/4]$ and $h>0$ (see Definition 4 in \cite{chola:14}), then from Proposition 2 in \cite{chola:14} it follows that $|\partial sup(f)|_d=0$, from where it follows that  \color{black} $int(\sup(f))\neq \emptyset$. Let $x\in int(sup(f))$, then we can remove from $int(sup(f))$ an arbitrary small open cone $T$ centred at $x$, whose interior is included in $int(\sup(f))$. Define $f_\delta=cf\mathbb{I}_{T^c}$ as being $c$ a normalizing constant, and $T^c$ the complement of the set $T$. Clearly $\textrm{reach}(sup(f_\delta))=0$.  Let us assume now that $\textrm{reach}(sup(f))=0$. Let $\eps>0$ and $K_1=K_1(\eps)$ a compact set such that $\int_{K_1^c}f(x)dx<\eps/4$.  Let $\kappa>0$ small enough such that 
 	$$\int f(x)\mathbb{I}_{f^{-1}([0,\kappa))}dx<\epsilon/4.$$
 	Define	$K=K_1\cap \overline{f^{-1}([\kappa,+\infty))}$ and $\xi(x)=d(x,K^c)\mathbb{I}_K(x)$. Let $\gamma$ small enough such that $|K\setminus \xi^{-1}([\gamma,+\infty))|_d<\eps/2$.  Let $g$ a $\mathcal{C}^2$ function such that $\sup_{x\in K} |\xi(x)-g(x)|<\gamma/2$. Let $S_\gamma:=g^{-1}([\gamma/2,\infty))$ then $S_\gamma\subset  K$, is compact and $\mathcal{C}^2$, and then its boundary is has positive reach,  (see \cite{tha08}). The density $cf\mathbb{I}_{S_\gamma}$ fulfills the desired properties, where the constant $c$ is chosen to integrate 1. 
 \end{proof}

 \section{A small simulation study}\label{simus}

 \subsection{Example 1}
 In this first example, to study the performance of $\hat{r}_n$, we considered the two dimensional sets,
 $$M_r= \{(x,y) \in \mathbb{R}^2 :  r^2 < x^2 + y^2 < r^2 + 1/\pi \},$$ 
 for  $r \in \{ 0.25, 0.5 \}$. Then, $\textrm{reach}(M_r)=r$. For each $r$ fixed, we drawn a sample of $n$ points uniformly distributed on $M_r$, for $n \in \{500,750,1000,1250,1500\}$. The values of $\hat{r}_n$ are computed using  $\eps_n = (\max_i\min_{j\neq i} ||X_i-X_j||)^{1/2}$, for a constant $c>(4/\pi)^{1/2}$. 
 The whole procedure is replicated $100$ times. The mean, median and standard deviation of these replications are shown in Table \ref{estim}. As is observed in both scenarios,  the performance of our estimator improves when the sample sizes increases. In addition, the deviation is smaller for $r=0.5$ than for $r=0.25$.  
 There is a small overestimation of the reach, expected according to Theorem  \ref{auxth}, which is smaller for $r=0.5$. 
 

 \begin{table}
 	\caption{Mean, median and standard deviation over $100$ replication, of $\hat{r}_n$ for different values of the reach $r$ and $n$.}
 	\label{estim}
 	\footnotesize
 	\begin{center}
 			\begin{tabular}{c|cccccc|c}\toprule
 				$n$ & \multicolumn{6}{c}{Inner radius} & $\eps_n$  \\
 				& \multicolumn{3}{c}{ $r$=0.25} & \multicolumn{3}{c}{ $r$=0.5}   & \\
 				\cmidrule(lr){2-4} \cmidrule(lr){5-7} 
 				&mean & median & sd &mean & median & sd & \\
 				500&   0.262&0.265&0.005&   0.503&0.504&0.002&     0.44 \\
 				750&   0.258&0.261&0.003&   0.502&0.503&0.002&     0.41\\
 				1000&   0.256&0.258&0.002&   0.502&0.502&0.001&       0.40\\
 				1250&   0.256&0.258&0.002&   0.502&0.500&0.001&         0.39 \\
 				1500&   0.255&0.255&0.002&   0.501&0.500&0.001&        0.37 \\
 				\bottomrule
 			\end{tabular}
 	\end{center}
 \end{table}

 \subsection{Bias correction }
 \label{sesgo}
 The overestimation of the reach is clearly established by Theorem  \ref{auxth}, and is also confirmed by the results in Table \ref{estim}. This problem becames worse in the case of manifold-valued data, where we conjecture that it is  due to the fact that $d_{G_n}$ tends to underestimate $d_M$ when $M$ is a lower-dimensional manifold.  \color{black} Thus, we consider  a bias correction, following  the proposal given in \cite{arias2019} for volume estimation, adapted for our problem. It worth to be mention that the proposal in \cite{arias2019} has been shown  to be minimax optimal under the assumption that the data is uniformly distributed on the manifold.  However, the consistency of the bias correction goes far beyond the scope of this manuscript. We show the performance through a simulation study.

 We split the set $\mathcal X_n$ into two subsets $\mathcal X_1$ and $\mathcal X_2$ of sizes $n_1$ and $n_2$, respectively.
 
 \begin{enumerate}
 	\item Calculate $\hat{r}_{n_1}$ based on the set $\mathcal X_1$.
 	\item Calculate $\hat p_n$, defined by 
 	\begin{multline*}
 		{n_2 \choose 2}^{-1}\Big \vert \Big\{(X_i,X_j) \in \mathcal X_2: i \neq j , \Vert X_i - X_j\Vert < 2 \hat{r}_{n_1} \ \mbox{ but } \\   d_{G_{n_2}}(X_i,X_j) >  2\hat r_{n_1}(1 - \eps_{n_2}^2)\text{arcsin}\Big(\frac{\Vert X_i-X_j\Vert  }{2\hat r _{n_1} }\Big)\Big\}\Big \vert.
 	\end{multline*}
 	where $|\cdot|$ denotes the cardinality of the set.
 	\item Output: $\hat r_n = [(1-\hat p_n) \vee 1/2]\hat{r}_{n_1}$.
 \end{enumerate}

 \subsubsection{Example 2}
 
 In this example, the set is  $M=\{ (x,y) \in \mathbb{R}^2 : x^2 + 4y^2 = 1, x\geq 0\}$, which is a manifold with boundary; see Figure \ref{reael}. It is easy to see that $\textrm{reach}(M)=0.25$. We compare the performance of our estimator and the estimator proposed in \cite{aa19}. Its first approach requires to know the tangent spaces  at the sample points.
 More precisely, given an iid sample  $\mathcal{X}_n \subset M$, $M$ being  a $d$-dimensional manifold in $\mathbb{R}^D$, the authors proposes the estimator $\tilde{r}$ given by
 
 \begin{equation}
 	\label{amm}
 	\tilde{r}= \inf_{X_i \neq X_j \in \mathcal{X}_n } \frac{\Vert X_j-X_i \Vert^2 }{ 2d \left( X_j-X_i, T_{X_i} M \right)},
 \end{equation}
 where  $T_{X_i} M$ is the $d$ affine linear space, tangent to $M$ at $X_i$.
 
 
 If  $M$ is unknown, then $T_{X_i} M$ must be estimated. Proposition 6.1 in \cite{aa19} provides bounds for the stability of the estimator $\tilde{\tau}$ when the tangent spaces are perturbed.
 Let  $TM= \{\widehat{T_{X_i}M}\}_{X_i \in \mathcal{X}_n}$ be a family of $d$ affine linear spaces, let 
 
 \begin{equation}
 	\label{amari}
 	\tilde{r} (\mathcal{X}_n, T)= \inf_{X_i \neq X_j \in \mathcal{X}_n } \frac{\Vert X_j-X_i \Vert^2 }{ 2 d\left( X_j-X_i, \widehat{T_{X_i}M} \right)}.
 \end{equation}

 \begin{figure}[ht]
 	\begin{center}
 		\includegraphics[height=6.6cm,width=7.2cm]{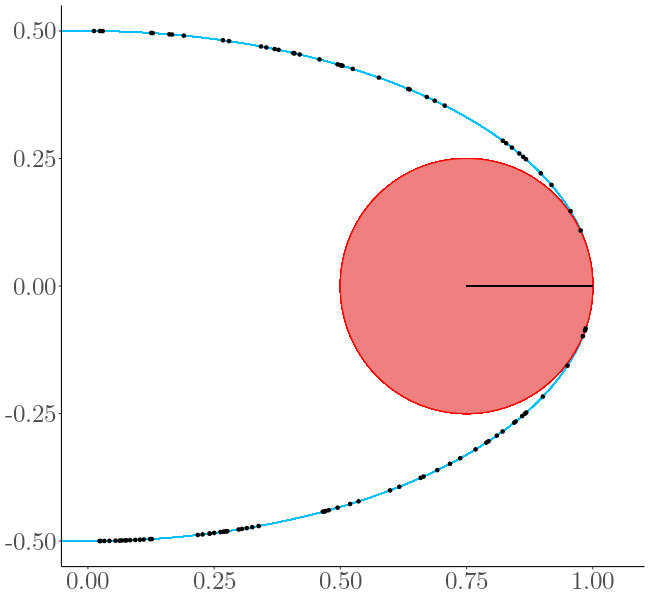}
 	\end{center}
 	\caption{ The half-ellipse $M=\{ (x,y) \in \mathbb{R}^2 : x^2 + 4y^2 = 1,\,\, x \geq 0\}$, whose reach is 0.25, and a sample on $M$.}
 	\label{reael}
 \end{figure}
 
 To compare with the proposals given in \cite{aa19},  we  consider $S_1=\{S_{1,X_i}\}_{X_i\in\mathcal{X}_n}$, defined for each $X_i\in\mathcal{X}_n$, as follows:
 $S_{1,X_i}$ is the estimation of $T_{X_i}M$, by means of local principal components, see  \cite{aa21}.

 We draw samples of size $n$ in the half-ellipse, for  $n \in \{400,600\}$. The samples were generated as follows: first we draw a sample $(X,Y)$ with standard normal bivariate distribution, then  we take $(|X|,Y)$. Finally we project $Z= (|X|,Y)$ onto the ellipse  considering $Z / \Vert Z \Vert_e$ with $\Vert (x,y) \Vert_e= \sqrt{ x^2+ 4 y^2}$.

 We perform 100 replications. For each of them we calculate 
 1) $\tilde{r}_{1,n}=\tilde{r} (\mathcal{X}_n, S_1)$ based on equation (\ref{amari}), 
 2) our estimator, $\hat{r}_{1,n}$, without bias correction, and 
 3) $\hat{r}_{2,n}$, the estimator with the bias correction proposed in Section \ref{sesgo}. We take $\eps_n = (\max_i\min_{j\neq i} ||X_i-X_j||)^{1/2}$, while we follow the suggestion in \cite{aa19} considering only those pairs of points whose distance is at least $\delta$ with $\delta$ of order 
 $\log (n) \big /n$.
 
 The boxplots  for these estimators are given in Figure \ref{rech1}.The average differences between the true tangent subspaces 
 $TM$  and their approximations $S_{1}$ (i.e.,  $\textrm{Error}(k)= \max_{X_i \in \mathcal{X}_n} \Vert T_{X_i}M - S_{1,X_i} \Vert_{op}$, where $k$ stands for the $k$--th replicate) is $0.078$ and $0.054$ for $n=400$ and $n=600$ respectively.
 Figure \ref{rech1} shows how the bias is reduced when we use the bias-corrected estimator $\hat{r}_{2,n}$, which also performs better than   $\tilde{r}_{1,n}$.

 \begin{figure}[ht]
 	\begin{center}
 		\includegraphics[height=7cm,width=12cm]{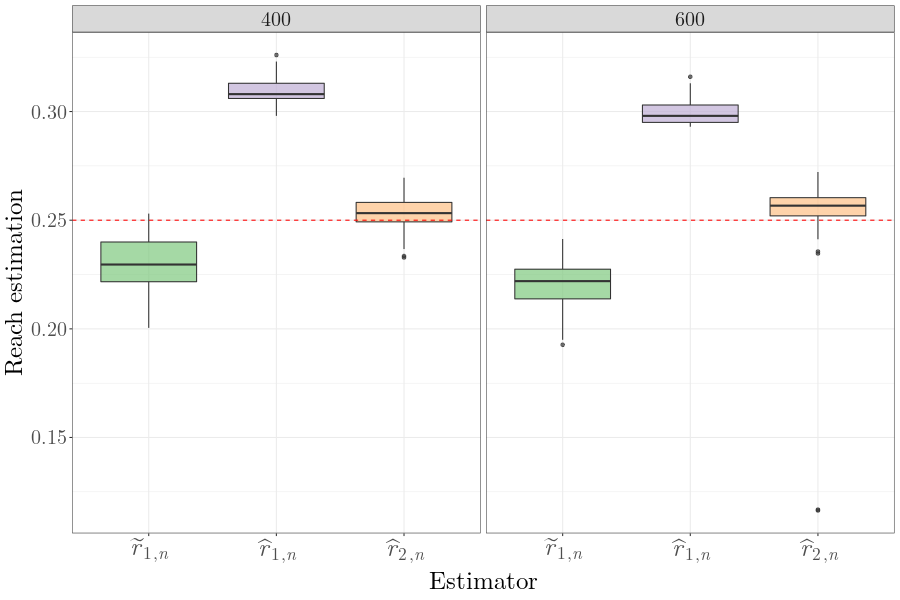}
 	\end{center}
 	\caption{ Estimators boxplots for $\tilde{r}_{1,n}= \tilde{r} (\mathcal{X}_n, S_1)$, $\hat{r}_{1,n}$  without bias correction and  $\hat{r}_{2,n}$ the estimator with the bias correction. }
 	\label{rech1}
 \end{figure}

 				\section{Concluding remarks}
 				
 				We deal with an important problem in set estimation related to geometric measure theory: the estimation of the reach, $r_0$, of a set $M \subset \mathbb R^d$, which includes the case where $M$ is a manifold.
 				A universally consistent estimator of $r_0$, in the sense that no assumptions are required except for $r_0$ being positive for  is proposed. 
 				We show that based on a finite sample with density $f$ w.r.t. the Lebesgue measure, it is not possible to determine if the reach of the support of $f$ is zero or not.
 				The consistency result is related to the convergence to zero of  $d_H(M, \mathcal{X}_n)$ . We conjecture that given any consistent estimator of the reach, $r_n$, and given any sequence $\beta_n\to 0$,  it is possible to find a set  (depending on $\beta_n$), with positive reach,  for which, $r_n$ converge to the reach at a rate slower than $\beta_n$. 
 				
 				However, under the weak additional assumption of standardness,  we provide rates of convergence for the proposed estimator. 
 				For the case where $M$ is a manifold, we adapt our procedure by adding a bias correction method.  An alternative  way to deal with this problem is to consider the estimator $\tilde{r}_n=(1-\alpha\eps_n)\hat{r}_n$, but the optimal choice of the parameter $\alpha$ should be addressed. 
 				
 				The results obtained in our small simulation section are promising. However, a more extended simulation study, to compare with other proposals in the literature is an interesting problem to be addressed.

 				\section*{Appendix}
 				
 				\subsection{Proof of Theorem \ref{consist}}
 				
 				\begin{proof} 
 					To prove the first inequality in \eqref{eqth}, let $X_i,X_j\in \mathcal{X}_n$ such that $||X_i-X_j||<2r_0 $. Let $n$ be large enough such that $(1+4\tau_n)\leq (1+\eps_n^2)$. From \eqref{reach2}
 					\begin{multline*}
 						d_{G_n}(X_i,X_j)\leq (1+4\tau_n)d_M(X_i,X_j)\leq 2r_0 (1+4\tau_n)\text{arcsin} \Big(\frac{\Vert X_i-X_j\Vert}{2r_0 }\Big)\\
 						\leq 	 2r_0(1+\eps_n^2)\text{arcsin} \Big(\frac{\Vert X_i-X_j\Vert  }{2r_0 }\Big).
 					\end{multline*} 
 					Then $\hat{r}_n\geq  r_0$.
 					To prove the second inequality in \eqref{eqth}, we have to prove that for all $x,y\in M$ such that $||x-y||\leq 2\hat{r}_n(1-\eps_n) $,
 					\begin{equation*}
 						d_M(x,y)\leq 2\hat{r}_n(1-\eps_n) \text{arcsin} \Big(\frac{\Vert x-y\Vert}{2\hat{r}_n(1-\eps_n) }\Big).
 					\end{equation*}
 					Since $r_0\leq \hat{r}_n$, it is enough to verify previous inequality for all $x,y\in M$ such that $||x-y||\geq 2r_0(1-\eps_n)$.   This condition will be used later on in the proof, to guarantee that the value of $n$ for which previous inequality holds, does not depend on the pair $x,y$.

 					Denote by 
 					$$\gamma_n=\Big(1- \frac{1}{24}\Big(\frac{\pi \eps_n}{2\textrm{reach}(M)}\Big)^2\Big).$$
 					Let $X_i,X_j\in \mathcal{X}_n$ be the closest points, wrt the euclidean distance, to $x$ and $y$, respectively.   From $\delta_n/\eps_n^3\to 0$, it follows that, for $n$ large enough, 
 					\begin{multline*}
 						||X_i-X_j||\leq ||X_i-x||+||x-y||+||y-X_j||\leq 2\delta_n+2\hat{r}_n(1-\eps_n)<2\hat{r}_n.
 					\end{multline*}
 					Then, for $n$ large enough
 					\begin{equation}\label{eq0}
 						d_{G_n}(X_i,X_j)\leq 2\hat{r}_n(1+\eps_n^2)  \text{arcsin} \Big(\frac{\Vert X_i-X_j\Vert}{2\hat{r}_n  }\Big).
 					\end{equation}
 					From Equation \eqref{auxeq1},  (recall that  $d_{G_n}(x,y)$ is the distance in the graph,  including $x$ and $y$ as vertices)
 					\begin{equation} \label{eq1}
 						\gamma_nd_M(x,y)\leq d_{G_n}(x,y)\leq d_{G_n}(x,X_i)+d_{G_n}(X_i,X_j)+d_{G_n}(X_j,y).
 					\end{equation}
 					Since $||x-X_i||<\delta_n<\eps_n$, $d_{G_n}(x,X_i)=||x-X_i||<\delta_n$. In the same manner, $d_{G_n}(y,X_j)<\delta_n$.  Since $\gamma_n\to 1$ we can take $n$ large enough such that $2\delta_n/\gamma_n<3\delta_n$, then from \eqref{eq0} and \eqref{eq1},  
 					\begin{equation*}\label{eq2}
 						d_M(x,y)\leq 3\delta_n+ 2\hat{r}_n\frac{(1+\eps_n^2)}{\gamma_n}\text{arcsin} \Big(\frac{\Vert X_i-X_j\Vert}{2\hat{r}_n }\Big).
 					\end{equation*}
 					Let $x,y$ such that $||x-y||<2\hat{r}_n(1-\eps_n)$, then
 					\begin{multline*}
 						\frac{||X_i-X_j||}{2\hat{r}_n  }\leq \frac{||x-y||+2\delta_n}{2\hat{r}_n }=\frac{||x-y||}{2\hat{r}_n(1-\eps_n) }+\frac{||x-y||}{2\hat{r}_n }\Big(1-\frac{1}{1-\eps_n}\Big)+\frac{\delta_n}{ \hat{r}_n}=\\
 						\frac{||x-y||}{2\hat{r}_n(1-\eps_n) }-\eps_n\frac{||x-y||}{2\hat{r}_n (1-\eps_n)}+\frac{\delta_n}{ \hat{r}_n}.
 					\end{multline*}
 					Using that $\text{arcsin}(a-b)\leq \text{arcsin}(a)-\text{arcsin}'(a-b)b$, for all $a,b>0$ such that $0\leq a-b<1$,  (which follows from the fact that $\text{arcsin}'(t)>1$ for all $0<t<1$), for $a=||x-y||/(2\hat{r}_n (1-\eps_n))$ and $b=\eps_na-\delta_n/\hat{r}_n$, we get
 					\begin{multline} \label{eq000}
 						d_M(x,y)\leq 2\hat{r}_n\frac{(1+\eps_n^2)}{\gamma_n}\Big[\text{arcsin}(a)-\text{arcsin}'(a-b)b\Big]+3\delta_n=\\
 						2\hat{r}_n(1-\eps_n)\text{arcsin}(a)+  2\hat{r}_n\Big(\frac{(1+\eps_n^2)}{\gamma_n}-1+\eps_n\Big)\text{arcsin}(a)\\
 						-2\hat{r}_n\frac{(1+\eps_n^2)}{\gamma_n}\text{arcsin}'(a-b)\Big(\eps_n a-\frac{\delta_n}{\hat{r}_n}\Big)+3\delta_n. 
 					\end{multline}
 					Using that $a<1$, $\hat{r}_n\geq r_0>0$, and $\gamma_n\to 1$, we get that, for $n$ large enough,
 					\begin{multline} \label{eq01}
 						2\hat{r}_n\frac{(1+\eps_n^2)}{\gamma_n}\text{arcsin}'(a-b)\frac{\delta_n}{\hat{r}_n}\leq 	3\text{arcsin}'(1-\eps_n+\delta_n/r_0)\delta_n \leq 	 C_1\frac{\delta_n}{\sqrt{\eps_n}}\leq \eps_n^2,
 					\end{multline}
 					$C_1$ being a positive constant. In addition, for $n$ large enough, using that $\hat{r}_n\leq diam(M)$, $\text{arcsin}(a)\leq \pi/2$, and $\gamma_n\to 1$, 
 					\begin{equation}\label{eq02}
 						\hat{r}_n\Big(\frac{(1+\eps_n^2)}{\gamma_n}-1\Big)\text{arcsin}(a)\leq \frac{diam(M)\pi}{4}(1-\gamma_n+\eps_n^2)\leq C\eps_n^2.
 					\end{equation}
 					$C$ being a positive constant. Then, using \eqref{eq01} and \eqref{eq02} in \eqref{eq000}, and the fact that $\delta_n/\eps_n^3\to 0$, we get that, for $n$ large enough,
 					\begin{multline*}
 						d_M(x,y)\leq 2\hat{r}_n(1-\eps_n)\text{arcsin}(a)+  C'\eps_n^2+\\
 						2\hat{r}_n\eps_n\Big[\text{arcsin}(a)-\frac{(1+\eps_n^2)}{\gamma_n}\text{arcsin}'(a(1-\eps_n)+\delta_n/\hat{r}_n)a\Big],
 					\end{multline*}
 					$C'$ being a positive constant. To prove the second inequality in \eqref{eqth}, it is enough to prove that there exists $k>0$ such that for $n$ large enough,
 					\begin{equation}\label{eqth3}
 						\text{arcsin}(a)-\frac{(1+\eps_n^2)}{\gamma_n}\text{arcsin}'(a(1-\eps_n)+\delta_n/\hat{r}_n)a<-k.
 					\end{equation}
 					
 					Observe that $k=0$ is not enough to conclude  because, if that is the case it may happen that $$\text{arcsin}(a)-\frac{(1+\eps_n^2)}{\gamma_n}\text{arcsin}'(a(1-\eps_n)+\delta_n/\hat{r}_n)a\to 0^-$$ 
 					as $n\to \infty$ at a  faster rate  than $\epsilon_n$, which does not guarantee that 
 					$d_M(x,y)\leq 2\hat{r}_n(1-\eps_n)\text{arcsin}(a)$ due to the $C'\eps_n^2$ term.
 					
 					From $||x-y||\geq r_0(1-\eps_n)$ and $\hat{r}_n\leq diam(M)$, we know that $r_0/diam(M)<a$, also $a<1+b$.
 					Let us bound,
 					\begin{multline*}
 						\text{arcsin}(a)-\frac{(1+\eps_n^2)}{\gamma_n}\text{arcsin}'(a(1-\eps_n)+\delta_n/\hat{r}_n)a\leq \\
 						\text{arcsin}(a)-\frac{(1+\eps_n^2)}{\gamma_n}\text{arcsin}'(a(1-\eps_n))a=:f_n(a).
 					\end{multline*}

 					Let us consider the function $g(t)=\text{arcsin}(t)-\text{arcsin}(t)'t$, which fulfills $g(0)=0$ and  $g'(t)<0$ for all $0<t<1$. From $r_0/diam(M)<a$ it follows that 
 					$g(a)\leq g(r_0/diam(M))<0$.
 					The functions $f_n$ converges uniformly to $g$, on any closed interval containing $a$. This entails that $f_n(a)\leq g(r_0/diam(M))/2$, for $n$ large enough, from where it follows  \eqref{eqth3}.
 					
 				\end{proof}
 				
 				\subsection{Proof of Theorem \ref{auxth}}
 				
 				This proof is based on the following Lemma, which is a modified version of Main Theorem A in \cite{bern} for sets of positive reach. The proof follows the same ideas used to prove Main Theorem A, so we will provide a sketch.

 				\begin{lem} \label{th0} Let $M\subset \mathbb{R}^d$ be compact,  such that $\emph{reach}(M)=r_0>0$. Let  $\mathcal{X}_n\subset M$ a set of $n$ points. We are given a graph $G$ on $\mathcal{X}_n$ and positive real numbers $\lambda_1,\lambda_2$. We also refer to positive real numbers $\eps_{min}$, $\eps_{max}$, and $\delta$. Suppose that,
 					\begin{enumerate}
 						\item The graph $G$ contains all edges $xy$ of length $||x-y||\leq \eps_{min}$.
 						\item All edges of $G$ have length $||x-y||\leq \eps_{max}$.
 						\item The set $\mathcal{X}_n$ fulfills $d_H(\mathcal{X}_n,M)\leq \delta$.
 						\item $M$ is geodesically convex.
 						
 						\hspace{-1cm} Then provided that
 						\item $\eps_{max}<2r_0$,
 						\item $\eps_{max}\leq (2/\pi)r_0\sqrt{24}\lambda_1$, 
 						\item $\delta\leq \lambda_2\eps_{min}/4$,
 					\end{enumerate}
 					it follows that inequalities
 					\begin{equation}\label{eqlem}
 						(1-\lambda_1)d_M(x,y)\leq d_{G}(x,y)\leq (1+\lambda_2)d_M(x,y)
 					\end{equation}
 					are valid for all $x,y\in M$.
 				\end{lem}
 				
 				\begin{proof} The second inequality in \eqref{eqlem} is proven  as the second inequality in Main Theorem A in \cite{bern}, which is based on Theorem 2 in \cite{bern}. The proof of Theorem 2 in \cite{bern} does not require smoothness assumptions on $M$. To prove the first inequality, let us consider a path $x_0x_1\dots x_p$ on $G$ connecting $x$ and $y$ (i.e., $x_0=x$ and $x_p=y$.) From conditions 2, 5 and 7 $||x_i-x_{i+1}||\leq 2r_0$ and $||x_i-x_{i+1}||\leq (2/\pi)r_0\sqrt{24}\lambda_1$. From \eqref{reach2}, we get that $2r_0\sin(d_M(x,y)/2r_0)\leq ||x-y||$, for all $x,y$ such that $||x-y||\leq 2r_0$ (this is Lemma 3 in \cite{bern}). Thus
 					$$d_M(x_i,x_{i+1})\leq (\pi/2)||x_{i}-x_{i+1}||\leq r_0\sqrt{24}\lambda_1.$$
 					We can then apply Corollary 4 in \cite{bern} with $\lambda=\lambda_1$ (its proof depends only on Lemma 3). Then,
 					$d_M(x_i,x_{i+1})\leq (1-\lambda_1^{-1})||x_i-x_{i+1}||$. Lastly,
 					$$d_M(x,y)\leq (1-\lambda)^{-1}||x_0+x_1||+\dots+(1-\lambda)^{-1}||x_{p-1}-x_p||=(1-\lambda_1)^{-1}d_G(x,y).$$
 				\end{proof}
 				
 				\noindent \textbf{Proof of Theorem \ref{auxth}}	
 				
 				\begin{proof}
 					Let us verify condition 1 to 6 in Lemma \ref{th0}. By construction of $G_n$, condition 1 and 2 are guaranteed with  $\eps_{max}=\eps_{min}=\eps_n$. Also $\delta_n=d_H(\mathcal{X}_n,M)$, guarantee condition 3.  Conditions 4 and 5 are fulfilled by hypotheses. Lastly, let us define 
 					$$\lambda_1:=\frac{1}{24}\Big(\frac{\pi \eps_n}{2r_0}\Big)^2\quad \text{ and } \quad \lambda_2:=4\tau_n,$$
 					then, conditions 6 and 7 are fulfilled, from where it follows \eqref{auxeq1}.
 				\end{proof}
 				
 				\section*{Acknowledgements} This research has been partially supported by  grant FCE-1-2019-1-156054, ANII, Uruguay.
 				The constructive comments and criticisms from two anonymous referees are gratefully acknowledged.

 			\end{document}